\documentclass{amsart}

\usepackage{amsthm,amsmath,amssymb}
\newtheorem{theorem}{Theorem}[section]
\newtheorem{lemma}[theorem]{Lemma}
\newtheorem{proposition}[theorem]{Proposition}
\newtheorem{corollary}[theorem]{Corollary}
\newtheorem{conjecture}[theorem]{Conjecture}

\theoremstyle{definition}
\newtheorem{definition}[theorem]{Definition}
\newtheorem{problem}[theorem]{Problem}
\newtheorem{example}[theorem]{Example}
\newtheorem{remark}[theorem]{Remark}

\numberwithin{equation}{section}

\newcommand{\Dast}[1]{\Delta^{\ast}(#1)}
\newcommand{\mD}[1]{\min\Delta(#1)}

\newcommand{\vo}{\mathsf{v}}
\newcommand{\ro}{\mathsf{r}}
\newcommand{\Zo}{\mathsf{Z}}
\newcommand{\Lo}{\mathsf{L}}
\newcommand{\qo}{\mathsf{q}}

\newcommand{\Do}{\mathsf{D}}

\newcommand{\ac}{\mathcal{A}}
\newcommand{\bc}{\mathcal{B}}
\newcommand{\fc}{\mathcal{F}}
\newcommand{\lc}{\mathcal{L}}
\newcommand{\ic}{\mathcal{I}}

\newcommand{\dc}{\mathcal{D}}

\newcommand{\uc}{\mathcal{U}}

\newcommand{\s}{\sigma}

\newcommand{\N}{\mathbb{N}}

\newcommand{\Z}{\mathbb{Z}}

\newcommand{\red}{\textrm{red}}
\newcommand{\fin}{\textrm{fin}}

\DeclareMathOperator{\ord}{ord}
\DeclareMathOperator{\lcm}{lcm}

\begin{document}
\title[Results and problems on sets of lengths of Krull monoids]{Some recent results and open problems on sets of lengths of Krull monoids with finite class group}
\author{W.~A. Schmid}

\email{schmid@math.univ-paris13.fr}
\address{Universit\'e Paris 13, Sorbonne Paris Cit\'e, LAGA, CNRS, UMR 7539, Universit\'e Paris 8, F-93430, Villetaneuse, France}

\thanks{Supported by the ANR project Caesar, project number ANR-12-BS01-0011.}

\begin{abstract}
Some of the fundamental notions related to sets of lengths of Krull monoids with finite class group are discussed,  and a survey of recent results is  given. These include the elasticity and related notions, the set of distances, and the structure theorem for sets of lengths. Several open problems are mentioned. 
\end{abstract}

\maketitle

\section{Introduction}
\label{sec_int}

Krull monoids are a central structure in factorization theory.
On the one hand, many structures of interest such as maximal orders of algebraic number fields and more generally Dedekind domains  are Krull monoids; we give some more examples in Section \ref{sec_prel}.
On the other hand, Krull monoids are by definition the class of monoids  one gets by considering the monoids whose arithmetic is given   by direct restriction of the arithmetic of a `surrounding' factorial monoid. Thus, there is also a purely intrinsic reason why they are a very natural type of monoid in this context, and this might be part of the reason why they arise in various areas.  

The investigation of the lengths of factorizations, that is the number of irreducible factors in the factorizations, is a central subject in factorization theory.
One reason for considering lengths is that the length is a simple and natural parameter  of a factorization, while still containing interesting information. There are other, more technical reasons, that are explained later. 

The idea of this survey article is to give some insight into current research on sets of lengths of Krull monoids, with an emphasis on the case of finite class group and each class containing a prime divisor. By `current' we roughly mean obtained during the last decade, or put differently since the publication of Geroldinger and Halter-Koch's monograph \cite{geroldingerhalterkochBOOK}, which covered this subject in detail (see especially Chapters 6 and 7).  

The scope is quite narrow and even in this narrow scope we do not attempt to be complete. Rather, the aim is to convey via discussion of selected subjects some of the main trends in recent research on this subject and to highlight some problems that might be interesting avenues for future research. In this vein, some effort is made to explain the \emph{why} and not only the what. For the most part, this survey does not contain proofs of the results we mention. However, proofs of some basic constructions and lemmas are included, on the one hand since sometimes the details of these proofs are relevant for the discussion and on the other hand to convey the type of arguments used.   

No attempt is made to faithfully recount the history of the subject. Of course, we try to attribute correctly the main results we discuss, but we also often make reference to secondary sources or even give none at all when we give a proof; this is the case especially for some basic results and constructions that are very widely known and used, and that sometimes exist in numerous slightly different versions in the literature. Except for Proposition \ref{prop_3u}, none of the results in this survey is new.

\section{Preliminaries}
\label{sec_prel}

We denote by $\N$ the set of positive integers and by $\N_0$ the set of non-negative integers. Intervals are intervals of integers, that is for real numbers $a,b$ we have $[a,b]= \{z \in \Z \colon a \le z \le b \}$.

For subsets $A, B $ of the integers we denote by $A + B = \{a+b  \colon a \in A , \, b \in B\}$ the sum of the sets  $A$ and $B$. For $k$ an integer we denote by $k \cdot A = \{ka \colon a \in A\}$ the dilation of $A$  by $k$.

In general we follow the notation and conventions of \cite{geroldingerhalterkochBOOK} and \cite{geroldinger_lecturenotes} where more detailed information could be found; the former gives an in-depth treatment of  factorization theory as a whole, the latter gives an introduction to the aspects most relevant to this survey, that is factorizations in Krull monoids and the associated zero-sum problems.

\subsection{Monoids, factorizations, sets of lengths}

In this paper, a monoid is a commutative, cancelative semigroup with identity, 
which we usually simply denote by $1$. We typically use multiplicative notation for monoids.
The multiplicative semigroup of non-zero elements of an integral domain is a good example to keep in mind.
Let $(H, \cdot)$ be a monoid. We denote by $H^{\times}$ the set of invertible
elements of $H$; we call the monoid reduced if $1$ is the only invertible element. 
By  $\ac (H)$ we denote the set of irreducible elements of $H$, also called atoms, 
that is the elements $a \in H \setminus H^{\times} $ such that $a=bc$ implies that  $b$ or $c$ is invertible. 
Moreover, we recall that an element $a$ is called prime if $a \mid bc$ implies that $a \mid b $ or $a \mid c$. 
Every prime  is irreducible; the converse is not necessarily true.  

We denote by $H_{\red}= H/H^{\times}$ the reduced monoid associated to $H$.
We say that elements $a, b \in H$ are associated, in symbols $a\simeq b$, if $a = \epsilon b$ 
with an invertible element $\epsilon \in H^{\times}$.

A monoid $F$ is called  free abelian if there exists a subset $P$ (of prime elements)
 such that every $a \in F$ has a unique representation of the form
\[
a = \prod_{p \in P} p^{\vo_p(a), } \text{ where } \vo_p(a) \in \N_0 \ \text{ with }  \mathsf v_p(a) = 0  \text{
for all but finitely many }  p \in P.
\]
We  use the notation $\fc(P)$ to denote the free monoid with $P$ as set of prime elements. 
We call  $|a| = \sum_{p \in P} \vo_p (a)$ the length of $a$.

The  monoid  $\Zo (H) = \fc \bigl(
\ac(H_\red)\bigr)$  is called the   factorization
monoid of $H$, and  the monoid homomorphism
\[\pi \colon \Zo (H) \to H_{\red}\]
induced by $\pi (a) = a$ for each  $a \in \ac(H_\red)$
is  the  called factorization homomorphism  of $H$. 

For $a \in H$, 
\[
 \Zo(a)= \pi^{-1} (a H^{\times})
\]
is the  set of factorizations  of  $a$
and 
\[
\Lo (a) = \bigl\{ |z| \colon  z  \in \Zo(a) \bigr\} \subset \N_0
\]
is the  set of lengths of $a$.
The above definition of the set of factorizations of $a$ is a formalization of what one could describe informally as
 the set of distinct (up to ordering and associates) factorizations of $a$ into irreducibles.  

In the present survey, we essentially exclusively deal with lengths of factorizations, and thus we are mainly interested in $\Lo(a)$. An alternate description for $\Lo(a)$, for $a \in H \setminus H^{\times}$,  is that it is the set of all $l$ such that there exist $u_1,\dots, u_l \in \ac(H)$ with $a= u_1 \dots u_l$; and setting $\Lo(a)=\{0\}$ for $a \in H^{\times}$. 

Moreover, we set $\lc(H)= \{\Lo(a) \colon a \in H  \}$ the system of sets of lengths of $H$.

\subsection{Abelian groups and zero-sum sequences}

We denote abelian groups additively. Mainly we deal with finite abelian groups. 
Let $(G, +, 0)$ be an abelian group. Let  $G_0 \subset G$ be a subset. 
Then $[G_0] \subset G$ denotes the subsemigroup generated by $G_0$, and $\langle G_0 \rangle \subset G$ denotes the
subgroup generated by $G_0$.  
A family of non-zero elements  $(e_i)_{i \in I}$ of $G$ is said to be  independent if, for $m_i\in \Z$, 
\[
\sum_{i \in I} m_ie_i = 0  \text{ implies }  m_ie_i = 0 \text{ for all }  i \in I.
\]
The tuple $(e_i)_{i \in I}$ is called a basis if $(e_i)_{i \in I}$ 
is independent and the elements $e_i$ generate $G$ as a group. 

For $n \in \mathbb N$, let
$C_n$ denote a cyclic group with $n$ elements. Suppose $G$ is finite. 
For $|G| > 1$, there are uniquely determined integers $1 < n_1 \mid \ldots \mid n_r$ such that 
\[
G \cong C_{n_1} \oplus \ldots \oplus C_{n_r}.
\]
We denote by $\ro(G)= r$ the rank of $G$ and by $\exp(G)= n_r$ the exponent of $G$. 
If $|G| = 1$, then $\ro (G)= 0$ and $\exp(G)=1$.
A group is called a $p$-group if the exponent is a prime-power. 
 
We set $ \Do^{\ast} (G) =1 + \sum_{i=1}^r (n_i-1)$; the relevance of this number is explained at the end of this subsection.

For $(G,+)$ an abelian group, and $G_0 \subset G$, we consider $\fc(G_0)$. It is common to call an element $S \in \fc(G_0)$ a sequence over $G_0$, and to use some terminology derived from it. In particular, divisors of $S$ are often called subsequences of $S$ and the neutral element of $\fc(G_0)$ is sometimes called the empty sequence.  

By definition 
\[
S = \prod_{g \in G_0} g^{\vo_g(S)} \] 
where  $\vo_g(S) \in \N_0$  with  $ \vo_g(S) = 0 $ 
for all but finitely many  $ g \in G_0$, and this representation is unique. 
Moreover, $S= g_1  \dots g_{|S|}$ with $g_i \in G_0$ for each $i \in [1, |S|]$ that are uniquely determined up to ordering. 

Since the set $G_0$ is a subset of a group, it makes sense to  consider the sum of $S$, that is 
\[\s(S)=  \sum_{g \in G_0} \vo_g(S)g= \sum_{i=1}^{|S|} g_i.  \] 
The sequence $S$ is called a zero-sum sequence if $\s(S)= 0 \in G$. A zero-sum sequence is called a minimal zero-sum sequence if it is non-empty and each proper subsequence is not a zero-sum sequence. 

The set of all zero-sum sequences over $G_0$ is denoted by $\bc(G_0)$; it is a submonoid of $\fc(G_0)$. 
The irreducible elements of $\bc(G_0)$ are the minimal zero-sum sequences; for brevity we denote them by $\ac(G_0)$ rather than by $\ac(\bc (G_0))$.

The Davenport constant of $G_0$, denoted by $\Do(G_0)$, is defined as 
\[\sup \{|A| \colon A \in \ac(G_0) \}.\]
It can be shown in general that $\Do(G_0)$ is finite if $G_0$ is finite (see \cite[Theorem 3.4.2]{geroldingerhalterkochBOOK}); 
in the special case that  $G_0$ is a subset of a finite group, or more generally  contains only elements of finite order, it however follows just by noting that in a minimal zero-sum sequence no element can appear with a multiplicity larger than its order.

For $G$ a finite abelian group, one has $\Do(G) \ge \Do^{\ast} (G)$. 
Equality is known to hold for groups of rank at most two and for $p$-groups. However, for groups of rank at least four it is known that the inequality is strict for infinitely many groups. We refer to \cite[Chapter 5]{geroldingerhalterkochBOOK} and \cite{geroldinger_lecturenotes} for more information on the Davenport constant  in the context of  factorization theory.

\subsection{Krull monoids and transfer homomorphisms}
\label{transfer}

We recall some basic facts on Krull monoids. For a detailed discussion on Krull monoids we refer to the relevant chapters of  Halter-Koch's monograph \cite{halterkoch98book} or again \cite[Chapter 2]{geroldingerhalterkochBOOK}.

There are several equivalent ways to defined a Krull monoid; the one we use is well-suited for the current context.   A monoid $H$ is called a Krull monoid if it admits a divisor homomorphism into a free monoid. 
This means there is some free monoid $\fc(P)$ and a monoid homomorphism $\varphi: H \to \fc(P)$ such that $a \mid b$ if and only if $\varphi(a) \mid \varphi(b)$. Thus, the arithmetic of a Krull monoid is directly induced by the one of a free, and thus factorial, monoid. 

There is an essentially unique `minimal' free monoid with this property, which is characterized by the property that for each $p \in P$ there exist $a_1, \dots, a_k \in H$ such that  
$p = \gcd(\varphi(a_1), \dots, \varphi(a_k))$. 

One calls a  divisor homomorphism $\varphi: H \to \fc(P)$ with the additional property, 
for each $p \in P$ there exist $a_1, \dots, a_k \in H$ such that  
$p = \gcd(\varphi(a_1), \dots, \varphi(a_k))$, a divisor theory. The elements of $P$ are called  prime divisors. 

Every Krull monoid admits a divisor theory, which is unique up to isomorphism. 
More specifically, a divisor theory is given by the  map from  $H$ to $ \ic_v(H)$, the monoid of divisorial ideals, mapping each element to the principal ideal it generates. This is indeed a free monoid in the case of Krull monoids as every divisorial ideal is in an essentially unique way, the product (in the sense of divisorial ideals) of divisorial prime ideals.  

Another characterization for Krull monoids is that they are completely integrally closed and $v$-noetherian, that is they satisfy the ascending chain condition on divisorial ideals. 

For  $\varphi: H \to \fc(P)$ a divisor theory, the group $G=\qo(\fc(P))/ \qo(\varphi(H))$ is called the class group of $H$. We denote the class containing some element $f$ by $[f]$; moreover, we use additive notation for the class group.  The set $G_P= \{[p] \colon p \in  P \} \subset G$ is called the set of classes containing prime divisors. The set $G_P$ generates $G$ as a semi-group; any  generating subset of $G$ can arise in this way.

Let $\tilde{\beta}: \fc(P) \to \fc(G_P)$ be the surjective monoid homomorphism induced  by $p \mapsto [p]$ for $p \in P$.  

One can see that the image of $\tilde{\beta} \circ \varphi $ is $\bc(G_P)$, and  
$\beta =  \tilde{\beta} \circ \varphi: H \to \bc(G_P)$ is called the block homomorphism. 

The block homomorphism is the archetypal example of a transfer homomorphism.
A monoid homomorphism $\theta: H \to B$ is called a transfer homomorphism if it has the following properties: 
\begin{itemize}
\item $B = \theta(H)B^{\times}$ and $\theta^{-1}(B^{\times}) = H^{\times}$. 
\item If $u \in H$, $b,c \in B$ and $\theta (u)= bc$, then there exist $v,w \in H$ such that $u= vw$, $\theta(v) \simeq b $ and $\theta(w) \simeq c $. 
\end{itemize}

An important property of transfer homomorphism is that $\Lo(a) = \Lo(\theta(a))$ for each $a \in H$, and $\lc(H)= \lc(B)$. Thus, a transfer homomorphism allows to transfer questions on sets of lengths from a monoid of interest $H$ to a simpler auxiliary  monoid $B$. The notion transfer homomorphism  was introduced by Halter-Koch \cite{halterkoch97};  an early formalization of the block homomorphism, in the context of rings of algebraic integers, was given by Narkiewicz \cite{narkiewicz79}.

\subsection{Examples of Krull monoids and related structures}

We gather some of the main examples of structures of interest to which the results recalled in this survey apply, 
that is structures that are Krull monoids or structures that admit a transfer homomorphism to a Krull monoid, which then usually is a monoid of zero-sum sequences. 

Before we start, we recall that a domain is a Krull domain if and only if its multiplicative monoid is a Krull monoid, as shown by Krause \cite{krause89}. Thus, we include Krull domains in our list of Krull monoids without further elaboration of this point. Moreover, we recall that Dedekind domains and more generally integrally closed noetherian domains are Krull domains (see, e.g., \cite[Section 2.11]{geroldingerhalterkochBOOK}). 

The following structures are Krull monoids. 
\begin{itemize}
\item Rings of integers in  algebraic number fields and more generally holomorphy rings in global fields (see, e.g., \cite{geroldingerhalterkochBOOK}, in particular Sections 2.11 and 8.9). 
\item Regular congruence monoids in Dedekind domains, for example the domains mentioned above (see, e.g., \cite{geroldingerhalterkoch04cong} or \cite[Section 2.11]{geroldingerhalterkochBOOK}).  
\item Rings of polynomial invariants of finite groups (see, e.g., \cite[Theorem 4.1]{cziszteretal}.
\item Diophantine monoids (see, e.g., \cite{chapmankrauseoek02}).
\end{itemize}

Moreover, the monoid of zero-sum sequences over a subset $G_0$ of an abelian group is itself a Krull monoid; the embedding $\bc(G_0) \hookrightarrow \fc(G_0)$ is a divisor homomorphism. 

Moreover, semi-groups of isomorphy classes of certain modules (the operation being the direct sum) turn out to be Krull monoids in various cases. There are many contributions to this subject; we refer to the recent monograph of Leuschke and Wiegand \cite{leuschkewiegand} for an overview.  We mention, specifically, a recent result by Baeth and Geroldinger \cite[Theorem 5.5]{baethgeroldinger14}, yielding a Krull monoid with cyclic classgroup such that each class contains a prime divisor (earlier example often had infinite class groups).

In addition to those examples of Krull monoids, there are structures that while not Krull monoids themselves, for example as they are not commutative or not integrally closed, still admit a transfer homomorphism to a Krull monoid. Hence their system of sets of lengths is that of a Krull monoid. 

We recall two recent results; the first is due to Smertnig \cite[Theorem 1.1]{smertnig13}, the second due to Geroldinger, Kainrath, and Reinhart \cite[Theorem 5.8]{geroldingerkr} (their actual result is more general).  

\begin{itemize}
\item 
Let $\mathcal{O}$ be a holomorphy ring in a global field  and let $A$ be a central simple algebra over this field. For $H$ a classical maximal $\mathcal{O}$-order of $A$ one has that if  every stably free left $H$-ideal is free, then there is a transfer-homomorphism from $H\setminus \{0\}$ to the monoid of zero-sum sequence over  a ray class group of $\mathcal{O}$, which is a finite abelian group.

\item Let $H$ be a seminormal order in a holomorphy ring of a global field with principal order $\widehat{H}$ such that the natural map $\mathfrak{X} (\widehat{H}) \to \mathfrak{X} (H)$ is bijective and there is an isomorphism  between the $v$-class groups of $H$ and $\widehat{H}$. Then there is a transfer-homomorphism from $H\setminus \{0\}$ to the monoid of zero-sum sequence over this $v$-class group, which is a finite abelian group.
\end{itemize}

In general we formulate the results we recall for Krull monoids. However, in cases where it seems to cause too much notational inconvenience, we give them for monoids of zero-sum sequences only. 

\section{Some general results}

In this section we collect some general results, before we focus on the more specific context of Krull monoids with finite class group in the subsequent sections.

\begin{definition}
Let $H$ be a monoid. 
\begin{enumerate}
\item $H$ is called atomic if $|\Zo(a)|>0$ for each $a\in H$. 
\item $H$ is called factorial if $|\Zo(a)|=1$ for each $a\in H$.
\item $H$ is called half-factorial if $|\Lo(a)| = 1 $ for each $a \in H$. 
\item $H$ is called an FF-monoid if $1 \le |\Zo(a)|< \infty $ for each $a \in H$. 
\item $H$ is called a BF-monoid if $1 \le |\Lo(a)|< \infty $ for each $a \in H$. 
\end{enumerate}
\end{definition}

The definition directly implies that all these monoids are atomic; a factorial monoid is half-factorial; an FF-monoid is a BF-monoid. 
It is not hard to see that  a Krull monoid is an FF-monoid, and thus a BF-monoid.  

Sets of lengths are subsets of the non-negative integers. However, sets of lengths containing $0$ or $1$ are very special. We make this precise in the following remark. 

\begin{remark}
Let $H$ be a monoid and let $a \in H$.
\begin{enumerate}
\item If $0  \in \Lo(a)$, then $\Lo(a)= \{0\}$ and $a \in H^{\times}$. 
\item If $1  \in \Lo(a)$, then $\Lo(a)= \{1\}$ and $a \in \ac(H)$. 
\end{enumerate}
\end{remark}

If $H$ is half-factorial, then $\lc(H) = \{\{ n \} \colon n \in \N_0 \}$. 
Going beyond half-factorial monoids, one might have the idea to relax the condition only slightly, 
say by imposing that each element has factorizations of at most two distinct lengths. 
However, this idea is infeasible, as the following lemma illustrates.

\begin{lemma}
\label{lem_add}
Let $H$ be an atomic monoid and let $a, b \in H$.  
Then $\Lo(a) + \Lo(b) \subset \Lo(ab)$. 
In particular, if $|\Lo(a)| > 1$, then $|\Lo(a^n)| > n$ for each $n \in \N$. 
\end{lemma}
\begin{proof}
Let $k \in \Lo(a)$ and $l \in \Lo(b)$. 
Let  $a = u_1 \dots u_k$ and $b = v_1 \dots v_l$ with irreducible $u_i, v_j \in \ac( H)$ for each $i \in [1,k]$ and $j \in [1,l]$. 
Then $ab = u_1 \dots u_k v_1 \dots v_l $ is a factorization of $ab$ of length $k+l$, 
and thus $k+l \in \Lo(ab)$.
The `in particular'-statement follows by an easy inductive argument, using the fact that 
for  $A,B \subset \Z$ of cardinality at least $2$, one has $|A+B|> |A|$ (in fact even $|A+B|\ge  |A| + |B|-1$). 
\end{proof}

We end this section by discussing some `extremal' cases for Krull monoids. 
The first result, in the context of rings of algebraic integers, goes back to Carlitz \cite{carlitz60}; for a proof in the context of monoids of zero-sum sequences, which suffices by the transfer result recalled in Section \ref{transfer}  see \cite[Theorem 3.4.11.5]{geroldingerhalterkochBOOK} or \cite[Proposition 1.2.4]{geroldinger_lecturenotes}.   

\begin{theorem}
\label{thm_carlitz}
Let $H$ be a Krull monoid such that each class contains a prime divisor. Then, $H$ is half-factorial if an only if its class group has order at most $2$. 
\end{theorem} 

The subsequent result is due to Kainrath \cite{kainrath99}.

\begin{theorem}
\label{thm_kainrath}
Let $H$ be a Krull monoid with infinite class group such that each class contains a prime divisor. Then, every finite subset of $\N_{\ge 2}$ is a set of lengths.  
\end{theorem} 

Thus for $H$ a Krull monoid with class group of order at most $2$, we have $\lc(H)= \{ \{ n\} \colon n \in \N_0\}$; for $H$  a Krull monoid with infinite class group such that each class contains a prime divisor we have $\lc (H ) =  \{\{0\}, \{1\}\}\cup \mathbb{P}_{\fin} (\N_{\ge 2})$, where $\mathbb{P}_{\fin}(\N_{\ge 2})$ denotes the set of all finite subsets of $\N_{\ge 2}$. 

For this reason we often restrict to considering the case of finite class groups of order at least $3$.

\section{Small sets}

As discussed, an atomic monoid that is not half-factorial always has arbitrarily large sets in its system of sets of lengths. 
One approach to understand the system of sets of lengths is to focus on `small' sets, that is those sets that  arise from factoring elements that are a product of only few irreducibles (their sets of lengths thus contain some small number). 

As an irreducible element $u$ has a unique factorization and $\Lo(u)=  \{1\}$, the next simplest case is to consider the product of two irreducibles. Studying the factorizations of $uv$, for $u,v \in \ac(H)$, turns out to yield interesting problems.

One natural question to ask is what other lengths can there be besides $2$ in a set of lengths. 
We start by recalling two basic constructions.  

\begin{lemma}
\label{lem_23}
Let $G$ be a finite abelian group of order at least $3$. 
\begin{enumerate}
\item Then $\{2, 3\} \in \lc(G)$.
\item If $g \in G$ is an element of order $n\ge 3$, then $\{ 2 , n \} \in \lc(G)$. 
\end{enumerate}
\end{lemma}
\begin{proof}
Let $g \in G$ be of order $n \ge 3$. 
Setting  $B = g^2(-2g) \cdot (-g)^2(2g)$ and noting $B= ((-g)g)^2 \cdot (-2g)2g$, it follows that $\Lo(B) = \{2,3\}$.
Note that $2g = - g$ holds for $n=3$, but this does not affect the argument. 
Moreover, setting $C=g^n (-g)^n$ and noting $C= ((-g)g)^n$  we see $\Lo(C)= \{2,n\}$.
 
It remains to show the first part in case there is no element of order at least $3$. 
If this is the case, there exist independent elements $(e_1,e_2)$ each of order $2$. 
We set $D= e_1^2 e_2^2 (e_1+e_2)^2$ and noting $D = (e_1e_2(e_1+e_2))^2$, it follows that $\Lo(D) = \{2,3\}$. 
\end{proof}

We note that  in some sense the simplest non-singleton set that can be a set of length, namely  $\{2,3\}$, is always in $\lc(G)$ for $|G| \ge 3$, but there is no absolute bound (that is one independent of $G$) on the size of elements in a set of lengths containing $2$. 
One natural question is to study this maximum size, for a given monoid $H$. 
Formally, one investigates $\max \{ \max \Lo (uv) \colon u,v \in \ac(H) \}$ or written differently $\max \left( \bigcup_{2 \in L, \, L \in \lc(H)} L \right)$. 
  
Similarly, one can consider the product of $3$ or more irreducibles. More generally, one considers the following quantities. 

\begin{definition}
Let $H$ be an atomic monoid. For $M \subset  \N_0$ let 
\[
\uc_M (H)= \bigcup_{M \subset L ,\, L \in \lc(H)}L.
\] 
Moreover, let $\lambda_M(H)= \min \uc_M (H)$ and $\rho_M(H)= \sup\uc_M (H)$.
\end{definition}

The case where $M$ is a singleton is of particular interest. 
For $k \in \N_0$, we write $\uc_{k}(H)$, $\lambda_{k}(H)$ and $\rho_{k}(H)$ for $\uc_{\{k\}}(H)$, $\lambda_{\{k\}}(H)$ and $\rho_{\{k\}}(H)$. These constants, especially $\rho_k(H)$ are those that received most interest so far.
The sets $\uc_k(H)$ were introduced by Chapman and Smith \cite{chapmansmith90a}, and the generalization $\uc_M(H)$ appeared in \cite{baginskietal13}.    

Moreover, the quantity $\rho(H) = \sup_{k \in \N} \rho_k(H)/k$ is called elasticity of the monoid, and it is also a classical constant in  factorization theory. The more common way to define it is as $\sup_{a \in H \setminus H^{\times}} \left(\sup \Lo(a)/\min \Lo(a) \right )$. We refer to \cite{anderson97} for an overview of classical results. 

We saw that $\uc_{0}(H)= \{0\}$ and $\uc_{1}(H) = \{1\}$. 
For $H$ a Krull monoid  with finite class group  $G$ such that each class contains a prime divisor, it is not difficult to determine $\rho_{\{k\}}(H)$ for even $k$; it is however a challenging problem for odd $k$. We show the former as part of the following well-known lemma, which we prove to give a general idea of the type of argument.  

\begin{lemma}
\label{lem_rho2}
Let $H$ be a non-factorial Krull monoid with   set of classes containing prime divisors $G_P$  such that the Davenport constant $\Do(G_P)$ is finite. 
\begin{enumerate}
\item $\rho_k(H) \le k \Do(G_P)/2$ for all $k \in \N$. 
\item If $G_P = -G_P$, then $\rho_{k+2}(H) \ge \rho_{k}(H) + \Do(G_P)$. In particular,  
\[
\rho_{k}(H) \ge 
\begin{cases} 
& \frac{k}{2} \Do(G_P) \quad k  \text{  even}  \\
& \frac{k-1}{2} \Do(G_P) +1 \quad k  \text{  odd}  
\end{cases}
\]
and $\rho_{2l}(H)= l \Do(G_P)$ for every $l \in \N_0$. 
\end{enumerate}
\end{lemma}
\begin{proof}
By the transfer results recalled in Section \ref{transfer} we can consider the problem in $\bc(G_P)$. We note that $\Do(G_P) \ge 2$ as the monoid is not factorial. 

1. Let $B \in \bc(G_P)$ with $k \in \Lo(B)$, say $B= U_1 \dots U_k $ with $U_i \in \ac(G_P)$ for each $i \in [1,k]$. 
Let $B=V_1 \dots V_r$ with $V_j \in \ac(G_P)$ for each $j \in [1,r]$. 

First, suppose $0 \nmid B$. Then $|V_j|\ge 2$ for all $j\in [1,r]$, while
$|U_i | \le \Do(G_P)$ for all $i\in [1,k]$, whence $2r \le |B| \le k \Do(G_P) $. 
Thus $r \le k \Do(G_P)/2$. This shows that every element of $\Lo(B)$ is bounded 
above by $k \Do(G_P)/2$, showing the claim. 

Now, let $B= 0^v B'$ where $v \in \N$ and $0 \nmid B'$. Then $\Lo(B) = v + \Lo(B')$ and $k-v \in \Lo(B')$. 
Thus, $ \max \Lo(B') \le (k-v) \Do(G_P)/2$ and $\max \Lo(B) \le  v +  (k-v) \Do(G_P)/2 \le k \Do(G_P)/2$.  

2.  Let $U= g_1 \dots g_l \in \ac(G_P)$.
Then $-U \in \ac(G_P)$. We have $(-U)U = \prod_{i=1}^{l}(-g_i)g_i$ and $(-g_i)g_i \in \ac(G_P)$ for all $i \in [1,l]$, it follows that $l \in \Lo((-U)U)$ and $l \le \rho_2(H)$. Let us now assume $U$ has length $|U|= \Do(G_P)$; such a $U$ exists by definition of $\Do (G_P)$. 

Let $B\in \bc(G_P)$ with $\{k, \rho_k (H)\}  \subset  \Lo(B)$. 
Then, as $\Lo((-U)U) + \Lo(B) \subset \Lo((-U)UB)$, we have
$\{k+2, \rho_k (H) + \Do(G_P)\}  \subset  \Lo((-U)UB)$ and the claim follows. 

To get the `in particular'-claim it suffices to apply this bound repeatedly, starting from $\rho_0(H)=0$ and $\rho_1(H)=1$. 
\end{proof}

We focus on the case that every class contains a prime divisor.
Since  $G_P = - G_P$ is trivially true, in this case  $\rho_{k}(H)$ is determined for even $k$, and we now recall some results for the case that $k$ is odd.  

From the preceding lemma one has the inequality 
\begin{equation}
\label{ineq_rhok}
k \Do(G) + 1  \le  \rho_{2k+1}(G) \le  k \Do(G)+  \left \lfloor  \frac{\Do(G)}{2}\right \rfloor.
\end{equation}

By a result of Gao and Geroldinger \cite{gaogeroldinger09}  it is known that for cyclic groups equality always holds at the lower bound. 

\begin{theorem}
Let $H$ be Krull monoid with finite cyclic class group $G$ of order at least $3$ such that each class contains a prime divisor. Then $\rho_{2k+1}(H) = k |G| + 1$ for all $k \in \N_0$.
\end{theorem}

The proof uses results on the structure of long minimal zero-sum sequences over cyclic groups (`long' is meant in a relative sense), see \cite{savchevchen07,yuan07}. As can be seen from the proof of the preceding lemma, one of the factorizations that could lead to a larger value of $\rho_{2k+1}(H)$ would have to be composed of minimal zero-sum sequences of length `close' to $\Do(G)$. Having knowledge on the structure of such sequences, allows to analyze this situation in a more explicit way. 

However, the case of cyclic groups seems to be quite exceptional, and there are various results asserting even equality at the upper bound in the inequality above. 

We recall a recent result due to Geroldinger, Grynkiewicz, Yuan \cite[Theorem 4.1]{geroldingergrynkiewiczyuan}. Moreover, they conjectured that cyclic groups and the group $C_2^2$ are the only groups for which $\rho_{3}(G)=  \Do(G) + 1$. 

\begin{theorem}
Let $H$ be Krull monoid with class group $G$ such that each class contains a prime divisor. Suppose that $G \cong \oplus_{i=1}^r C_{n_i}^{s_i}$ with $1 < n_1 \mid \dots \mid n_r$  and $s_i \ge 2$ for each $i \in [1,r]$.
Then, for every $k \in \N$, 
\[
\rho_{2k+1}(H) \ge  (k-1) \Do(G)  + \Do^{\ast}(G)+  \left \lfloor  \frac{\Do^{\ast}(G)}{2}  \right \rfloor.
\]
In particular, if  $\Do^{\ast}(G) = \Do(G)$, then $\rho_{2k+1}(G) =  k \Do(G)  +   \lfloor  \frac{\Do(G)}{2} \rfloor $ for every $k \in \N$.
\end{theorem}
The point of considering $\Do^{\ast}(G)$ rather than $\Do(G)$ is that the former is explicitly known and one thus has explicit examples of minimal zero-sum sequences of the relevant length that can be used to construct examples. By contrast, $\Do(G)$ is in general not known, and thus knowledge on zero-sum sequence of this length can only be obtained by general considerations.

For other conditions that imply equality  at the upper bound in \eqref{ineq_rhok} see for example \cite[Theorem 6.3.4]{geroldingerhalterkochBOOK}. Indeed, Geroldinger, Grynkiewicz, Yuan \cite[Conjecture  3.3]{geroldingergrynkiewiczyuan} put forward the conjecture that for sufficiently large $k$ this equality always holds for non-cyclic groups. 
  
\begin{conjecture}
Let $H$ be Krull monoid with finite non-cyclic class group $G$ such that each class contains a prime divisor. Then there exists some $k^{\ast} \in \N$ such that for each $k \ge k^{\ast}$ one has 
\[
\rho_{2k+1}(H) =  k \Do(G) + \left \lfloor  \frac{\Do(G)}{2} \right \rfloor . 
\]
\end{conjecture}
To restrict to sufficiently large $k$ is certainly necessary, as the following result illustrates, see Geroldinger, Grynkiewicz, Yuan \cite[Theorem 5.1]{geroldingergrynkiewiczyuan}.

\begin{theorem}
Let $H$ be Krull monoid with class group $G$ such that each class contains a prime divisor. Suppose that $G \cong C_m\oplus C_{mn}$ with $m \ge 2$ and $n \ge 1$.
Then  
\[
\rho_{3}(H) =   \Do(G)  + \left \lfloor  \frac{\Do(G)}{2}\right  \rfloor  \text{ if and only if } n=1 \text{ or } n=m =2.
\]
\end{theorem}
The proof uses the fact that the structure of minimal zero-sum sequences of maximal length is known for groups of rank $2$ (see \cite{reiherB,GGGinverse3,WASinverse2}). To put this in context, we remark that to know the sequences of maximal lengths allows to exclude equality at the upper bound for most groups of rank $2$; to get further improved  upper bounds might need knowledge on the structure of long (yet not maximum length) minimal zero-sum sequences in addition, as known and used in the case of  cyclic groups.

This result allows to give examples of groups where the actual value of $\rho_3(G)$ can be neither the upper nor the lower bound in \eqref{ineq_rhok}. An example is $C_2 \oplus C_{2n}$ for $n \ge 3$; 
however, in line with the above mentioned conjecture, one still has equality of $\rho_{2k+1}(G)$ with the upper bound for $k \ge 2n-1$ (see \cite[Corollary 5.3]{geroldingergrynkiewiczyuan}). 

Very recently Fan and Zhong \cite[Theorem 1.1]{fanzhong} made considerable progress towards the above-mentioned conjecture. In particular, they verified it under the assumption that $\Do(G)= \Do^{\ast}(G)$. 

\begin{theorem}
Let $H$ be Krull monoid with finite non-cyclic class group $G$ such that each class contains a prime divisor. Then there exists some $k^{\ast} \in \N$ such that for each $k \ge k^{\ast}$ one has 
\[
\rho_{2k+1}(H) \ge  (k-k^{\ast}) \Do(G) + k^{\ast} \Do^{\ast}(G)  +  \left \lfloor  \frac{\Do^{\ast}(G)}{2} \right \rfloor. 
\]
In particular, if $\Do(G)= \Do^{\ast}(G)$, then $\rho_{2k+1}(H) =  k \Do(G)   +  \left \lfloor  \frac{\Do(G)}{2} \right \rfloor$ for $k \ge k^{\ast}$.
\end{theorem}

Having discussed $\rho_k(H)$ in some detail, we turn to the other constants. However, we see that in important cases the determination of $\uc_k(H)$ and $\lambda_k(H)$ can be reduced to the problem of determining $\rho_k(H)$. 
 
The following result is due to Freeze and Geroldinger \cite[Theorem 4.2]{freezegeroldinger08}; for another proof of this result due to Halter-Koch see \cite[Theorem 3.1.3]{geroldinger_lecturenotes}. 

\begin{theorem}
Let $H$ be a Krull monoid with finite class group such that each class contains a prime divisor. Then $\uc_{k}(H)$ is an interval for every $k \in \N$. 
\end{theorem}

Thus, in this case it suffices to determine $\lambda_k(H)$ and $\rho_k(H)$ to know $\uc_k(H)$.
Moreover, it is even possible (see \cite[Corollary 3.1.4]{geroldinger_lecturenotes}) to express (in this case) the constants $\lambda_k(H)$ in terms of $\rho_k(H)$. 

\begin{theorem}
Let $H$ be a Krull monoid with finite class group $G$ such that each class contains a prime divisor.
Then for every $k \in \N_0$ we have 
\[
\lambda_{k \Do(G)+j}(H) = 
\begin{cases} 2k  & \quad \text{ for } j=0 \\
2k +1 & \quad \text{ for } j \in [1, \rho_{2k+1}(H) - k \Do(G)] \\
2k +2 & \quad \text{ for } j \in [\rho_{2k+1}(H) - k \Do(G) +1, \Do(G)-1 ]  
\end{cases}
\]
\end{theorem}

We turn to results on $\uc_M(H)$ where $M$ is not a singleton. 
In view of the results above, we see that if $\min M$ and $\max M$ are too far apart then for $H$ a Krull monoid with finite class group the sets $\uc_M (H)$ will always be empty. 
Specifically, when $\min M=k$, then $\uc_M (H)$ is empty if $M$ contains some element greater than $\rho_k(H)$. 

Considering $\uc_{\{k, \rho_k(H)\}}(H)$ is thus an interesting extremal case. 
This problem was investigated recently by Baginski, Geroldinger, Grynkiewicz, Philipp \cite{baginskietal13}, with a focus on groups of rank two; again, it is important to know the structure of minimal zero-sum sequences of maximal length. 

We start by recalling an older result for cyclic groups and elementary $2$-groups (see \cite[Theorem 6.6.3]{geroldingerhalterkochBOOK}).  

\begin{theorem}
\label{thm_u2cycel2}
Let $H$ be a Krull monoid with finite class group such that each class contains a prime divisor.
Then, 
\(
\uc_{\{2, \rho_2(H)\}}(H)= \{2, \rho_2(H)\}
\)
if and only if the class group is cyclic or an elementary $2$-group.
\end{theorem}

For groups of rank $2$ the set $\uc_{\{2, \rho_2(H)\}}(H)$ is a lot larger as shown in \cite[Theorem 3.5]{baginskietal13}. 

\begin{theorem}
Let $H$ be a Krull monoid with class group $G \cong C_m \oplus C_{mn}$ where $m,n \in \N$ and $m\ge 2$ such that each class contains a prime divisor.
Then,
\[
\uc_{\{2, \rho_2(H)\}}(H)= 
\begin{cases}
\{2a \colon a \in [1,n]\} \cup \{\rho_2(H)\} &  \text{for } m = 2\\
[2, \rho_2(H)] &  \text{for }   m \in [3,4]\\
[2, \rho_2(H)] \setminus \{3\} &  \text{for }   m \ge 5  
\end{cases}
\]  
\end{theorem}
     
If the class group is a  group of rank greater than $2$, one faces the following problem. While one still knows $\rho_2(H)= \Do(G)$, one does in general not know $\Do(G)$ explicitly.  Thus, one also has only little knowledge on the form of minimal zero-sum sequences of maximal length. 

However, for most groups for which $\Do(G)= \Do^{\ast}(G)$ holds a description of $\uc_{ \{2, \rho_2(H)\}}(H)$ can still be obtained, as  more generally, $\uc_{\{2 , \Do^{\ast}(G)\}}(H)$ can be described almost completely for most groups of rank at least $3$. The following result was obtained in \cite[Theorem 4.2]{baginskietal13}.

\begin{theorem}
\label{thm_U2}
Let $H$ be a Krull monoid with class group 
$G \cong \oplus_{i=1}^r C_{n_i}$ where $1 < n_1 \mid \dots \mid n_r$ with  $r\ge 3$ and $n_{r-1} \ge 3$ such that each class contains a prime divisor.
Then,
\(
\uc_{\{2, \Do^{\ast}(G)\}}(H) \supset [2, \Do^{\ast}(G)].  
\)
In particular, if $\Do(G) = \Do^{\ast}(G)$, then $\uc_{\{ 2, \rho_2(H) \}} = [2, \rho_2(H)]$. 
\end{theorem}
 
We highlight the similarity to the results on $\rho_2(H)$, where also for general groups one resorted to $\Do^{\ast}(G)$ instead of $\Do(G)$.

We end this section with a small complement to the preceding theorem, investigating the relevance of the condition on $n_{r-1}$.

\begin{proposition}
\label{prop_3u}
Let $H$ be a Krull monoid with class group 
$G \cong  C_{2}^{r-1} \oplus C_{2n}$ with  $r\ge 3$ and $n \in \N$ such that each class contains a prime divisor.
Then,
\(
3 \in \uc_{\{2, \Do^{\ast}(G)\}}(H)  
\)
if and only if $n \ge 3$. 
\end{proposition} 
\begin{proof}
By the transfer results that we recalled in Section \ref{transfer} we can assume $H=\bc(G)$. Let $(e_1, \dots, e_{r-1}, f)$ be a basis of $G$ with $\ord e_i = 2$ for $1 \le i \le r-1$ and $\ord f = 2n$.

First, suppose $n \ge 3$. We note that the sequence $U = f^{2n-3} (f+e_1)^3 (f+e_2) (-f+e_1+\dots + e_{r-1})  e_3 \dots e_{r-1}$ is a minimal zero-sum sequence of length $\Do^{\ast}(G)$: 
the assertion on the sum and length are direct, and to see that it is minimal we note that $f^{2n-3} (f+e_1)^3 (f+e_2)$ has no non-empty subsequence with sum $0$, so 
that a zero-sum subsequence $T$ of $U$ has to contain one and then each element of $(-f+e_1+\dots + e_{r-1})  e_3 \dots e_{r-1}$, 
which implies that $T$ contains   $(f+e_1) (f+e_2)$ and thus must equal $U$ to get sufficiently many elements containing $f$.

We consider $(-U)U$. Of course it has factorizations of length $2$ and $\Do^{\ast}(G)$. It remains to show that it has a factorization of length $3$. 
To see this note that $(-U)U$ is equal to $V_1V_2 V_3$ with $V_1 = f^{2n-3} (f+e_1) (f+e_2) (f+e_1+\dots + e_{r-1})  e_3 \dots e_{r-1}$, $V_2= (-f)^{2n-5} (-f+e_1)^3 (-f+e_2) (-f+e_1+\dots + e_{r-1})  e_3 \dots e_{r-1}$, and $V_3 = (f+e_1)^2 (-f)^2$, and $V_1, V_2, V_3$ are minimal zero-sum sequences. 

For $n=1$ it is established in Theorem \ref{thm_u2cycel2} that $3 \notin \uc_{\{2, \Do^{\ast}(G)\}}(H) $. 
It remains to consider $n=2$. Note that in this case the exponent of $G$ is $4$, so $G$ is a $2$-group and $\Do(G)=\Do^{\ast}(G)$ (see Section \ref{sec_prel}). 
Assume for a contradiction there is a zero-sum sequence  $B$ over $G$ such that $\{2,3, \Do^{\ast}(G)\} \subset \Lo(B)$. 
As $\Do(G)=\Do^{\ast}(G)$ and $\{2, \Do^{\ast}(G)\} \subset \Lo(B)$, 
it follows that $B= U(-U)$ where $U$ is a minimal zero-sum sequence of length $\Do(G)=\Do^{\ast}(G)$ (see the proof of Lemma \ref{lem_rho2}). 
Since $3 \in \Lo(U(-U))$ it follows that $U(-U)= V_1V_2V_3$ with minimal zero-sum sequences $V_1, V_2, V_3$ and further for each $1 \le i \le 3$ we have $V_i =S_i(-T_i) $  with  
$U = S_1S_2 S_3 =T_1T_2T_3$. We note that none of the $S_i$ and $T_i$ is the empty sequence.
We have $\s(S_i) = \s(T_i)$ for each $1 \le i \le 3$ and moreover $\sigma(S_1)+ \sigma(S_2) +\sigma(S_3)= 0$. 

We claim that at least one of the elements $\sigma(S_1), \sigma(S_2) ,\sigma(S_3)$ has order $2$. Since $U$ is a minimal zero-sum sequence all three elements are non-zero, as sums of proper and non-empty subsequences of $U$.  
Denoting by $G[2]$ the subgroup of $G$ of elements of order at most $2$, we have $G/G[2]$ is a group of order $2$ and as the images of $\sigma(S_1), \sigma(S_2) ,\sigma(S_3)$  in $G/G[2]$ form a zero-sum sequence not all of them can be the non-zero element in $G/G[2]$. Consequently, at least one of the elements has order at most $2$ and as it must be non-zero it has order $2$, establishing the claim. 

Without loss of generality, we assume that $\sigma(S_3)=e$ has order $2$. 
If $\gcd (S_3,T_3 ) = 1$, then $S_3T_3  \mid U$. As $\s(T_3S_3)= 2 e = 0$, it follows that $S_3T_3= U$, that is $T_3 = S_1 S_2$ and $S_3 = T_1 T_2$.
Yet then $S_3(-T_3) = (S_1(-T_1))(S_2 (-T_2))$ contradicting the fact that $S_3(-T_3)$ is a minimal zero-sum sequence. 

Thus,   $\gcd (S_3,T_3 ) \neq 1$. This implies, as $S_3(-T_3)$ is a minimal zero-sum sequence, that $|S_3|=|T_3|=1$ and $S_3 = T_3 = e$.
  
If $\gcd (S_1, T_1) = \gcd (S_2, T_2) = 1$, then $S_2 = T_1$ and $S_1 = T_2$. As $\s(S_1)= \s(T_1)$, it follows that $\sigma(S_1) = \sigma (S_2)$ and thus $e = -2 \sigma(S_1) \in 2 \cdot G$. However, this is not possible, as a minimal zero-sum sequence of maximal length over a $2$-group must not contain an element from $2\cdot G$ (see \cite[Proposition 5.5.8]{geroldingerhalterkochBOOK}).  
Alternatively, one can argue that the image of $Ue^{-1}$ in $G/\langle e \rangle \cong C_2^r$ has to be a minimal zero-sum sequence, which is not possible as its length exceeds the Davenport constant of $C_2^r$. 

Thus, we get that   $\gcd (S_2, T_2) \neq 1$. As above we get that $|S_2|=|T_2|=1$. Yet then $|S_1(-T_1)| = 2|U| -4 = 2 \Do(G) - 4  > \Do(G)$, a contradiction. 
\end{proof}

The preceding results yield the following corollary. 

\begin{corollary}
Let $H$ be a Krull monoid with class group 
$G$ of rank $r  \ge 3$ such that each class contains a prime divisor. The following conditions are equivalent:
\begin{itemize}
\item $G$ is neither an elementary $2$-group nor of the form $C_2^{r-1} \oplus C_4$. 
\item \(3 \in \uc_{\{2, \Do^{\ast}(G)\}}(H). \)
\end{itemize}
\end{corollary}

We mention that this corollary allows to fill what we believe to be  a minor gap in the proof of \cite[Theorem 5.6]{baginskietal13}; 
it can be invoked there instead of  \cite[Theorem 4.2]{baginskietal13} (that is the result we recalled as Theorem \ref{thm_U2}).

\section{Distances}

In the preceding section we discussed how spread out sets of lengths can be, in the sense of comparing their extremal values. We now turn to the question how large distances there can be between adjacent elements of the sets of lengths. Moreover, considering distances also gives another measure for the complexity of a set of lengths; highly structured sets, such as arithmetic progressions, have few distinct distances even when the set itself might be large.   
 
\begin{definition} \
\begin{itemize}
\item Let  $A \subset \Z$. Then the  set of distances of $A$, denoted by $\Delta (A)$, is  the set of all differences between consecutive elements of $A$, formally, it is the set of all $d \in \N$ for which there exists $l \in
A$ such that $A \cap [l, l+d] = \{l, l+d\}$.  
\item For an atomic monoid $H$, we denote by 
\[
\Delta (H)  = \bigcup_{a \in H} \Delta \bigl( \mathsf L (a) \bigr) \subset \N
\]
 the  set of distances of $H$.
\end{itemize}
\end{definition}
It is sometimes common to denote, for $a \in H$, the set $\Delta(\Lo(a))$ by $\Delta(a)$. Since we only use it rarely, we do not use this abbreviation here.

If $H$ is a Krull monoid with finite  class group, then $ \Delta(H)$ is finite. 
More specifically and more generally, one has the following general bound (see, e.g., \cite[Theorem 3.4.11 and Theorem 1.6.3]{geroldingerhalterkochBOOK}.

\begin{lemma}
Let $H$ be a Krull monoid and let $G_P$ denote the set of classes containing prime divisors.
Then $\sup \Delta(H) \le \Do(G_P) -2$. 
\end{lemma}

In case the class group is infinite, $\Delta(H)$ can be infinite, too. In fact, if  each class contains a prime divisor then $\Delta(H)= \N$. (This is a direct consequence of Theorem \ref{thm_kainrath}, yet it is a much simpler result; indeed, we give a partial proof below.)  

The example recalled in Lemma \ref{lem_23} shows that for a Krull monoid where each class contains a prime divisor we always have $1 \in \Delta(H)$. 
Moreover, Geroldinger and Yuan \cite{geroldingeryuan12} showed that for these monoids $\Delta(H)$ is an interval. 

\begin{theorem}
Let $H$ be a Krull monoid with finite class group such that each class contains a prime divisor.
Then $\Delta(H)= [1, \max \Delta(H)]$. 
\end{theorem} 

Thus, in this important case the problem of determining $\Delta(H)$ is reduced to the problem of determining the maximum of this set. Before we discuss results towards this goal, we recall some well-known  constructions to get some rough insight into which size of $\max \Delta(H)$ one might expect (for further details see, e.g., \cite[Lemma 6.4.1]{geroldingerhalterkochBOOK}).

\begin{lemma}
Let $G = C_{n_1} \oplus \ldots \oplus C_{n_r}$ with $|G| \ge 3$ and $1 < n_1 \mid \dots \mid n_r$. 
Then 
\[
[1, n_r -2] \cup [1, -1+ \sum_{i=1}^{r}\left \lfloor\frac{n_i}{2}\right \rfloor ] \subset \Delta(G) 
\]
\end{lemma}
\begin{proof}
Let $e_1, \ldots, e_r \in G$ be independent with $\ord e_ i = n_i $ for each $i \in [1,r]$. 
Let $e_0 = k_1 e_1 + \ldots + k_r e_r$, where  $k_i \in \N_0$ and $2k_i \le \ord e_i$ for all $i \in [1, r]$. 
For 
\[
U = (-e_0) \prod_{i=1}^r e_i^{k_i},
\]
we have  $\mathsf L ((-U) U ) = \{2, k_1+ \dots + k_r + 1\}$. 
This yields a distance of  $-1 + k_1+ \dots + k_r$ (except if $k_1+ \dots + k_r= 1$). Since  
$k_1+ \dots + k_r$ can attain any value in $ [ 1, \sum_{i=1}^{r}\lfloor\frac{n_i}{2}\rfloor]$, 
we get  $ [ 1, -1 + \sum_{i=1}^{r}\lfloor\frac{n_i}{2}\rfloor] \subset \Delta(G)$ . 

Let $e \in G$ be non-zero. Then $ \Lo(e^n ((a-1)e)(-e)^{a-1}) = \{2, a\}$ for $a \in [2, \ord (e) ]$. 
This yields a distance of $a-1$. As there is an element of order $n_r$, we get $[1, n_r -2]$. 
\end{proof}

\begin{remark}
Since an infinite abelian torsion  group contains elements of arbitrarily large order or 
an infinite independent set, the above constructions show $\Delta(G)= \N$ for infinite torsion groups. 
\end{remark}

No element in $\Delta(G)$ larger than the ones given above is known.
The bound   $ \max \Delta(G )  \le \Do(G)-2$ shows that  for $G$ cyclic or an elementary $2$-group, indeed, there can be no larger element. Thus, one has the following result (see \cite[Theorem 6.4.7]{geroldingerhalterkochBOOK}). 

\begin{theorem}
\label{thm_dist2c}
For  $r \ge 2$ and $n \ge 3$ one has $\Delta(C_2^r) = [1, r-1]$ and $\Delta(C_n) = [1, n-2]$.  
\end{theorem}
These groups are in fact the only ones for which $ \max \Delta(G )  = \Do(G)-2$. A characterization of groups for which $ \max \Delta(G )  = \Do(G)-3$ was recently given by Geroldinger and Zhong \cite{geroldingerzhongCat}.

However, in general the following problem is wide open. 

\begin{problem}
Let  $G \cong C_{n_1} \oplus \dots \oplus C_{n_r}$ with $|G| \ge 3$ and $1 < n_1 \mid \dots \mid n_r$.
Is \[\max \Delta(G) = \max \left  \{n_r -2 ,  -1+ \sum_{i=1}^{r}\left \lfloor\frac{n_i}{2}\right \rfloor \right \}\ ?\] 
\end{problem}

We recall results that give upper-bounds on  $\max \Delta (H)$. 
It turned out that the following quantity is a useful tool to this end.
It was introduced in \cite{WAScat}. The problem of determining $\max \Delta (H)$ and problems of distances more generally are often studied in combination or even via a notion called catenary degree. 
The catenary degree is a notion of factorization theory that does not only take the length of factorizations into account, which is why we do not discuss it here.   

\begin{definition} Let $H$ be an atomic monoid. 
Let
\[
\daleth (H) = \sup \big\{ \min \bigl( \Lo (uv) \setminus \{2\} \bigr) \colon u,v \in \mathcal A(H) \big\},
\]
with the convention that $\min \emptyset = \sup \emptyset =0$.
\end{definition}
We point out that we again study sets of lengths of a product of two irreducible elements; other aspects of this problem were discussed in the preceding section. The following lemma is essentially a direct consequence of the definition. 

\begin{lemma} Let $H$ be an atomic monoid. Then
$\daleth(H) \le  2 + \sup \Delta(H)$.      
\end{lemma}
While equality does not always hold (for an example see below), it can be shown to hold for Krull monoids under certain assumptions on the class group. Informally, this then means that the largest possible distance is already attained in the sets of lengths of the product of two irreducible elements, which simplifies the task of actually determining this distance.    

The following result is a special case of \cite[Corollary 4.1]{WAScat}. 

\begin{theorem} 
Let $H$ be a Krull monoid with class group $G \cong  C_{n_1}
\oplus \ldots \oplus C_{n_r}$ where $1<n_1|\ldots |n_r$ and $|G|\ge 3$ such that each class contains a prime divisors. If 
 \[
\left\lfloor\frac{1}{2}\mathsf D (G)+1 \right \rfloor \le
           \max \left\{n_r,\,1+\sum_{i=1}^{r}\left \lfloor\frac{n_i}{2} \right \rfloor \right\}.
\]
Then
\(
\daleth (H) = 2 + \max \Delta(H).
\)
\end{theorem}
We discuss the technical condition. Since 
\[
1+\sum_{i=1}^{r} \Big\lfloor\frac{n_i}{2} \Big\rfloor =
\frac{1 + \mathsf{r}_2(G)+ \Do^{\ast} (G)}{2} ,
\]
where $\mathsf{r}_2(G)$ denotes the number of even $n_i$s,
it follows that if $\Do(G)=\Do^{\ast}(G)$, then
\[
\Big\lfloor\frac{1}{2}\mathsf D(G)+1 \Big\rfloor
\le 1+\sum_{i=1}^{r}\Big\lfloor\frac{n_i}{2} \Big\rfloor.
\]

We give an example where $\daleth (H) < 2 + \max \Delta (H)$. 
For details of the example see  \cite[Proposition 4.1.2]{geroldingerhalterkochBOOK}.

\begin{example}
\label{ex_n-r-1}
Let $G$ be an abelian group and  $r, \,n \in \N_{\ge 3}$ with $n \ne r+1$.
Let  $e_1,\dots , e_r \in G$ be independent elements with  $\ord e_i = n$
for all $i \in [1,r]$. We set $e_0 = -(e_1 + \ldots + e_r)$ and $G_0 =
\{e_0,e_1 \dots , e_r\}$. Then $\Delta (\bc(G_0)) = \{|n-r-1|\}$ yet $\daleth(\bc(G_0))= 0$. 
\end{example}
To see this note that the only minimal zero-sum sequences are $e_i^n$ for $i \in [0,r]$ and 
$W = \prod_{i=0}^r e_i$. To have a non-trivial relation, we at least need to have $W^n$, 
which factors also as $\prod_{i=0}^r e_i^n$.  

We continue with a bound on $\daleth(H)$; this is a special case of \cite[Theorem 5.1]{WAScat}. 

\begin{theorem}
Let $H$ be a Krull monoid with finite class group $G$ such that  each class contains a prime divisor. If $\exp (G) = n$ and $\mathsf r (G) = r$, then
\[
\daleth (H) \le    \max\left\{n,\quad \frac13\left(2\mathsf
D (G)+\frac{1}{2} rn+ 2^r\right)\right\}.
\]
\end{theorem}

In combination with the preceding result one obtains bounds for $\max \Delta(H)$ for various types of class groups. We formulate one explicitly.

\begin{corollary}
Let $H$ be a Krull monoid with finite class group $G \cong C_n^2$ with $n \ge 2$ such that each class contains a prime divisor. Then 
\[  \max \Delta (H) \le \frac{5n-4}{3}.\]
\end{corollary}

We recall that the lower bound for $\max \Delta (H)$ is $n-2$ for odd $n$ and $n-1$ for even $n$ whereas the simple upper-bound given by $\Do(C_n^2)-2$ is $2n-3$. 

We point out that for this problem knowledge of the structure of minimal zero-sum sequences of maximal length seems insufficient. The extremal known examples are attained by minimal zero-sum sequences of length about $\Do(C_n^2)/2$. 

Up to now we only discussed $\Delta(H)$, that is the collection of all distance that can occur in some monoid. 
It is also an interesting question to study $\Delta  ( \Lo(a)  ) $ for individual elements of $a \in H$. 
By definition it is clear that each $d \in \Delta(H)$ occurs in  $\Delta  ( \Lo(a)  ) $ for some $a \in H$.  
Yet, passing to  more than one distance, one gets interesting questions. 
For example, for distances  $d_1,d_2 \in H$ one can ask if there exists some $a \in H$ such that $d_1,d_2  \in \Delta(\Lo(a))$.
Or, for some fixed distance $d \in H$ one can ask what are all the other distances in the sets of lengths having $d$ as a distance; formally, one can study similarly to $\uc_k(H)$ the sets 
\[\bigcup_{a\in H,  \,   d   \in \Delta( \Lo(a)) } \Delta( \Lo(a)).\]
Recently, Chapman, Gotti, and Pelayo \cite{chapmanetal14} obtained the following result on this type of problem. 

\begin{theorem}
Let $H$ be a Krull monoid with cyclic class group of order $n \ge 3$, and let $a \in H$. 
If $n -2 \in  \Delta( \Lo (a))$, then $ \Delta( \Lo (a)) = \{n-2\}$.  
\end{theorem}
We recall that $n-2$ is the maximum of the set of distances for Krull monoid with cyclic class group $n$, assuming that each class contains a prime divisor. A similar result for elementary $2$-groups is also known, see \cite[Lemma 3.10]{WASsumset}.

\section{Large sets}

Sets of lengths can be arbitrarily large. However, one can show that they are not arbitrarily complicated, in a sense to be made precise. 

The construction we saw in Lemma \ref{lem_23}, when we recalled that there cannot be a global bound on the size of sets of lengths in non-half-factorial monoids, suggests that there is some additive structure to large sets of lengths. 
Indeed, this is the case for various classes of monoids. We recall the result and related relevant notions. 

\begin{definition}
A non-empty subset $L$ of $\Z$ is called an almost arithmetic multi-progression (AAMP for short) with bound $M \in \N_0$, difference $d \in \N$ and period $\dc$ (where $\{0, d\} \subset \dc \subset [0,d]$) if 
\[
L = y + (L' \cup L^{\ast} \cup L'') \subset   y +  \dc +  d \cdot \Z
\]
with $0 \in  L^{\ast} = [0, \max L^{\ast}] \cap  (\dc +  d \cdot \Z)$ and $L' \subset [-M, -1]$ and $L'' \subset \max L^{\ast} + [1,M]$.   
One calls $L^{\ast}$ the central part, and $L'$ and $L''$ the beginning and the end part, respectively. 
\end{definition}

The notion of AAMP turns out, as we see below, to be natural for describing sets of lengths of Krull monoids with finite class group, and also other monoids. Informally, one can imagine an AAMP as a union of several slightly shifted copies of an arithmetic  progressions where at the beginning and the end some elements might be removed.  
The definition of AAMP contains the following special cases.

\begin{definition} \ 
\begin{itemize}
\item an AAMP with bound $M=0$ is called an arithmetic multi-progression (AMP for short).  
\item an AAMP with period $\dc = \{0,d\}$ is called an almost arithmetic progression (AAP for short). 
\item an AAMP with bound $M= 0$ and period $\dc = \{0,d\}$ is called an arithmetic progression (AP for short).
\end{itemize}
\end{definition}
The notion of AP just recalled of course coincides with the usual notion of a finite arithmetic progression. 
The notion of arithmetic multi-progression should not be confused with that of multi-dimensional arithmetic progressions,  
which is typically defined as  a sumset of several arithmetic progressions.

Some care needs to be taken when saying that some set is or is not an AAMP. 
In fact, one has: 
\begin{itemize}
\item every non-empty finite set $L\subset \Z$ is an AAP for with  bound $\max L - \min L  $ (and period $\{0,1\}$).
\item every non-empty finite set $L\subset \Z$ is an AMP for with period $-\min L + L$.
\end{itemize}
Thus, it is crucial to restrict bound and period in some way to make saying that a set is an AAMP meaningful.

The importance of the notion of AAMP in this context is mainly due to the following result, a Structure Theorem for Sets of Lengths (STSL). This result is due to Geroldinger \cite{geroldinger88}, except that there  a slightly different notion of AAMP was used; the current version was obtained in \cite{freimangeroldinger00}.  

\begin{theorem}
\label{thm_STSL}
Let $H$ be a Krull monoid with finite class group. There is some $M  \in \N_0$ and  a non-empty finite set  $\Delta^{\ast} \subset \N$ such that 
for each $a \in H$ its set of lengths $\Lo(a)$ is a AAMP with bound $M$ and difference $d$ in $\Delta^{\ast}$.  
\end{theorem}

A crucial point in this result is that the bound and the set of differences depend on the monoid, and not on the element. Indeed, by the transfer results recalled in Section \ref{transfer} they depend on the class group or more precisely the subset of classes containing prime divisors, only. 

This result was generalized in several ways and is known to hold for various other classes of monoids, too (see \cite[Chapter 4]{geroldingerhalterkochBOOK}). Even sticking to Krull monoids it holds under the weaker condition that only finitely many classes contain prime divisors, or still weaker, that the Davenport constant of the set of classes containing prime divisors is finite (see Theorem \ref{thm_STSLfD}).

The Structure Theorem for Sets of Lengths raises various follow-up questions. On the one hand, it is a natural question to ask if this description is a natural one or if there could be a simpler one. On the other hand, the result contains a bound $M$ and a set of differences $\Delta^{\ast}$ and the question arises what are the actual values of these parameters. 
We discuss this in the remainder of this section.

\subsection{The relevance of AAMPs}

Realizations results for sets of lengths prove that in a certain sense Theorem \ref{thm_STSL} is optimal. 
We recall such a realization result from \cite{WASreal}; for earlier result of this form see \cite[Section 4.8]{geroldingerhalterkochBOOK}.

\begin{theorem}
\label{thm_real}
Let $M \in \N_0$ and let $\emptyset \neq \Delta^{\ast} \subset \N$ be a finite set. Then, there exists a Krull monoid $H$ with finite class group such that the following holds: 
for every set $L$ that is an AAMP with difference $d \in \Delta^{\ast}$  and bound $M$ there is some $y_{H,L}$ such that 
\[y + L  \in \lc(H)  \text{ for  all }   y \ge y_{H,L}.\] 
\end{theorem}

This result implies the existence of  Krull monoids with finite class group whose system of sets of lengths contains all possible sets whose maximum and minimum are not too far apart. (Though, this was known already earlier.) 

\begin{corollary}
Let $M \in \N_0$. Then, there exists a Krull monoid $H$ with finite class group such that \( L  \in \lc(H)\) for every $L \subset  \mathbb{N}_{\ge 2}$ with $\max L - \min L \le M$. 
\end{corollary}

In \cite{WASreal} some explicit conditions on the class group were obtained that guarantee that the above results hold. For example, it is known that $\bc(C_p^{r})$ for $p$ a prime greater than $5$ and $r \ge 21 (M^2 + \max \Delta^{\ast})$ fulfills the conditions of Theorem \ref{thm_real} and thus of  the corollary, too.
This motivates the following problem. 

\begin{problem}
Can one determine a function $f: \N \to \N$  such that for $G$ a finite abelian group with  $|G| \ge f(M)$ one has that  $\lc(G)$ contains each finite set $L\subset \N_{\ge 2}$  with $ \max L - \min L \le M $.  
\end{problem}

The author believes that such a function exists and a solution of this problem should be well within reach of current methods and results. The appeal of having such a result would be that it would give a precise way to express the informal idea that   $\lc(G)$ contains all possible sets that are `small' relative to $G$. 

The result that $\lc(G)$ for infinite $G$ contains every finite set $L\subset \N_{\ge 2}$ could be thought of as a limiting case of this result, for an infinite group every finite set $L\subset \N_{\ge 2}$ is `small.' In fact, a positive answer to this problem would even yield a proof of the result for infinite torsion groups. 

We do not recall a proof of Theorem \ref{thm_real} but still recall some simple constructions that show how AAMPs arise naturally in this context (cf. Lemmas \ref{lem_add} and \ref{lem_23}). 

\begin{lemma}
\label{lem_AP}
Let $g \in G$ be an element of order $n \ge 3$. 
Then $\Lo((g(-g))^{k n})$ is an AP with difference $n-2$ and length $k$, more specifically it is $2k+(n-2) \cdot [0,k]$. 
\end{lemma}
\begin{proof}
The only minimal zero-sum sequences over the set $\{-g,g\}$ are $(-g)g$, $g^n$, and $(-g)^n$.
The only factorizations of $(g(-g))^{k n}$ are thus $(g^n(-g)^n)^{k-j}   (g(-g))^{nj}$ for $j \in [0,k]$; their lengths are $2(k - j) + jn $. 
\end{proof}

Based on this lemma we give explicit examples of richer structures arising as sets of lengths; we choose to really fix some parameters to avoid confusion from having many parameters.  

\begin{example}
Let $e_1, e_2, g,h \in G$ be independent elements of order $2$, $2$, $10$ and $14$ respectively, then 
\[
\Lo((g(-g))^{10k}(h(-h))^{14k})= \{4k  \} \cup   (4k + 8 + 4 \cdot [0, 5k-4]) \cup \{24 k\}
\]
is an AAP with difference $4$ and bound $8$, and 
\[
\begin{aligned}
\Lo( (e_1e_2(e_1+e_2))^2 & (g(-g))^{10k}(h(-h))^{14k})=    \{4k +2, 4k+3 \}  \\
 & \cup   (4k + 10 +\{0,1\} +  4 \cdot [0, 5k-4]) \cup \{24k +2, 24k+3\}
\end{aligned}
\]
is an AAMP with difference $4$, period $\{0,1,4\}$ and bound $8$.
\end{example}

\subsection{Some special cases}

As discussed for a general result the notion of AAMP seems inevitable. 
However, for special classes of groups simpler descriptions can be obtained. 
This is of course the case for class groups $C_1$ and $C_2$ where the system 
of sets of lengths consists of singletons only (see Theorem  \ref{thm_carlitz}), 
but it is certainly  also the case for $C_2^2$ and $C_3$ where by Theorem \ref{thm_dist2c} 
one has that $\Delta(G)= \{1\}$, which implies that all sets are intervals. 

In recent work of Geroldinger and the author \cite{WAScharR2} a characterization 
of all groups was obtained for which the more restrictive notions AP, AAP, or AMP suffice to describe 
the system of sets of lengths of $\bc(G)$.
We recall the result.  (The definition and relevance of the set $\Dast{G}$, used in the result below, is recalled later in this section; 
the exact definition is not really crucial for the  result below, and it could be replaced by $[1, |G|]$ for example.) 

\begin{theorem}
Let $G$ be a finite abelian group.
\begin{enumerate}
\item The following statements are equivalent:
      \begin{itemize}
      \item All sets of lengths in $\lc (G)$ are arithmetical progressions.
      \item $G$ is cyclic of order $|G| \le 4$ or isomorphic to a subgroup of  $C_2^3$ or isomorphic to a subgroup of  $C_3^2$.
      \end{itemize}
\item The following statements are equivalent:
      \begin{itemize}
      \item There is a constant $M \in \mathbb N$ such that all sets of lengths in $\mathcal L (G)$ are {\rm AAP}s with bound $M$.
      \item $G$ is isomorphic to a subgroup of $C_3^3$ or isomorphic to a subgroup of $C_4^3$.
      \end{itemize}
\item The following statements are equivalent:
      \begin{itemize}
      \item All sets of lengths in $\lc (G)$ are AMPs with difference in $\Dast{G}$.
      \item $G$ is cyclic with $|G| \le 5$  or isomorphic to a subgroup of $C_2^3$ or isomorphic to a subgroup of $C_3^2$.
      \end{itemize}
\end{enumerate}
\end{theorem}

In several of these cases it is even possible to give a complete description of $\lc(G)$. 
We already discussed the first point several times; for the following ones see \cite[Theorem 7.3.2]{geroldingerhalterkochBOOK}, and for the last one \cite[Proposition 3.12]{WASsumset}.

\begin{proposition} \ 
\begin{enumerate}
\item $\lc (C_1) = \lc (C_2) = \big\{ \{m\} \colon m \in \N_0 \big\}$.
\item $\lc (C_3) = \lc (C_2 \oplus C_2) = \bigl\{ y
      + 2k + [0, k] \colon  y,\, k \in \N_0 \bigr\}$.
\item $\lc (C_4) = \bigl\{ y + k+1 + [0,k] \colon y,
      \,k \in \N_0 \bigr\} \,\cup\,  \bigl\{ y + 2k + 2 \cdot [0,k] \colon  y,\, k \in \N_0 \bigr\} $.
\item $\lc (C_2^3)  =  \bigl\{ y + (k+1) + [0,k] \colon  y \in \N_0, \ k \in [0,2] \bigr\}$ \newline
      $\quad \text{\, } \ \qquad$ \quad $\cup \ \bigl\{ y + k + [0,k] \colon  y \in \N_0, \ k \ge 3 \bigr\}
      \cup \bigl\{ y + 2k
      + 2 \cdot [0,k] \colon y ,\, k \in \N_0 \bigr\}$.
\item  $\lc (C_3^2) = \{ [2k, l] \colon k \in  \N_0, l \in [2k, 5k]\}$ \newline
 $\quad \text{\, } \ \qquad$ \quad $\cup \ \{ [2k+1, l] \colon k \in \N, l \in [2k+1, 5k+2] \} \cup \{ \{ 1\}  \}$.
\end{enumerate}
\end{proposition}

However, to obtain results of this complete form becomes quite difficult. We recall a quite precise yet not complete description for the  group of order $5$ from \cite{WAScharR2}.  

\begin{lemma}
Let $G$ be a cyclic group  of order $|G|=  5$. Then every $L \in \lc (G)$ has one of the following forms:
\begin{itemize}
\item $L$ is an arithmetical progression with difference $1$.
\item $L$ is an arithmetical progression with difference $3$.
\item $L$ is an  AMP with period $\{0,2,3\}$ or with period $\{0, 1, 3\}$.
\end{itemize}
\end{lemma}

\subsection{The set of differences}

The formulation of  the Structure Theorem of Sets of Lengths contains a set $\Delta^{\ast}$. 
We give an overview on the current knowledge about these sets.
Of course, given the way the result is phrased this set cannot be determined uniquely; 
for one thing, if some set $\Delta^{\ast}$ is admissible for some bound $M$, then any superset of it would work, too.

Yet, there is a natural choice for the set $\Delta^{\ast}$ in the STSL for Krull monoids with finite class group, it is 
\[\Dast{H} = \{\mD{S} \colon  S \subset H \text{ a divisor-closed submonoid with } \Delta(S)\neq \emptyset\}.\]
We recall that a submonoid $S \subset H$ is called divisor-closed if for each $s \in S$ every $a \in H$ with $a \mid s$ (in $H$) is in fact an element of $S$. 

The result holds true for this set and it can be shown that $\lc(H)$ contains AAMPs with difference $d$ for each $d \in \Dast{H}$, so that it is not ``too large.'' The details of the proof of the STSL provide further justification  for considering this set as the natural choice. 

It should be noted though that in general this is not  a minimal choice. 
If $L$ is an AAMP with  difference $d$, period $\dc$ and bound $M$, then $L$ is also an AAMP with difference $md$, period $\dc + d \cdot [0,m-1]$, and bound $M$. 
Thus, if the STSL  holds for some set $\Delta^{\ast}$ that contains elements  $d, d'$ with $d \mid d'$, then one could omit $d$ without effect on the result. 

Thus, one could in principle ``simplify'' the set $\Dast{H}$ by omitting elements that are  a divisor of an element already in the set. Yet doing so rather obscures the situation without yielding a true simplification.
 
Similarly, setting $D = \lcm \Dast{H}$ one can even replace the set of differences by a unique difference and get the following reformulation of the STSL.

\begin{corollary} 
\label{cor_STSL}
Let $H$ be a Krull monoid with finite class group. There is some $M  \in \N_0$ and some $D \in \N$  such that for each $a \in H$ its set of lengths $\Lo(a)$ is a AAMP with bound $M$ and difference $D$.
\end{corollary}

While somewhat simpler to state, this formulation captures the reality of the situation not as well as the common one.

By transfer results as recalled in Section \ref{transfer} one can get that  
\[
\Dast{H} =  \{\mD{G_0} \colon  G_0 \subset G_P, \, \Delta(G_0 )  \neq \emptyset \}
\]
where as usual $G_P \subset G$ denotes the subset of classes containing prime divisor and $G$ the class group. (Some extra care  is needed to check that divisor-closed submonoids actually are preserved in this way.)

For $|G|\ge 3$, one denotes by $\Dast{G}=  \{\mD{G_0} \colon  G_0 \subset G, \, \Delta(G_0 )  \neq \emptyset \}$; this  matches the usual convention that 
$\Dast{G} = \Dast{\bc(G)}$.

By Lemma \ref{lem_23} we know that $\mD{G} = 1$  for $|G| \ge 3$. Thus $1 \in \Dast{G}$. 
Moreover the following constructions of elements of $\Dast{G}$ are classical. 

\begin{lemma}
Let $G$ be a finite abelian group with $|G|\ge 3$. 
\begin{enumerate}
\item $[1, \ro(G)-1] \subset \Dast{G}$. 
\item $d-2 \in \Dast{G}$ for each $3 \le d \mid \exp(G)$.
\item $|n-r-1| \in \Dast{C_n^r}$ for $n \ge 2$, $r \ge 1$, and $n \neq r+1$.
\end{enumerate}
In particular, $\max \Dast{G} \ge \max \{\ro(G)-1, \exp(G)-2\}$.
\end{lemma}
\begin{proof}
We only give a sketch for details see \cite{geroldingerhalterkochBOOK}. 
For the first point,  let $d \in [2,r]$ and let  $e_1, \dots, e_d \in G $ be independent elements of the same order, which we denote by $n$; note that by the definition of the rank such elements exist. Further, let $e_0= \sum_{i=1}^n e_i$. It follows that $W_j = e_o^j \prod_{i=1}^d e_i^{n-j}$ for $j \in [1,n]$ and $e_i^n$ for $i \in [1,d]$ are the only minimal zero-sum sequences. 
One has $W_j W_k = W_{j+k} \prod_{i=1}^d e_i^{n}$  for $j+k \le n$, and 
$W_j W_k = W_{j+k-n} W_n$ for $j+k>n$ are the only non-trivial relations. 
The former relations yield a distance of $(d+1)-2 = d-1$.

For the second point, we consider the set $\{-g,g\}$ for an element of order $g$; cf. Lemma \ref{lem_AP}.

For the third point, we consider the example given in Example \ref{ex_n-r-1}.
\end{proof}

Recently, Geroldinger and Zhong \cite{geroldingerzhongDast} proved that in fact  the inequality above is an equality; partial results and relevant techniques appeared in various papers, including \cite{gaogeroldinger00,WASchar}. 

\begin{theorem}
Let $H$ be a Krull monoid with finite class group $G$.
\begin{enumerate}
\item If $|G|\le 2$, then $\Dast{H} = \emptyset$.
\item If $2 < |G|< \infty$, then $\max \Dast{H} \le \max \{\exp(G)-2, \ro(G)-1\}$. If every class contains a prime divisor then equality holds. 
\end{enumerate}
\end{theorem}
For the case of infinite class group it was proved by Chapman, Schmid, Smith \cite{WASdeltainf} that if each class contains a prime divisor then $\Dast{H}= \N$. 

For groups $G$ where the rank is large relative to the exponent the set $\Dast{G}$ is  completely determined by the preceding theorem.  

\begin{corollary}
Let $G$ be a finite abelian group. If $\ro(G)-1 \ge \exp(G) -2$, then $\Dast{G}= [1, \ro(G)-1]$.  
\end{corollary}
Moreover, directly from the above results, for $ \exp(G) -2 = \ro(G)$ the set $\Dast{G}$ must still be an interval, namely $[1, \exp(G)-2]$, yet for groups with $\ro(G) < \exp(G) -2$ the set  $\Dast{G}$ could have gaps. Indeed, it frequently does have gaps, as the result below shows (it is a direct consequence of \cite[Theorem 3.2]{WASchar} and \cite{geroldingerzhongDast}).

\begin{theorem}
Let $H$ be a Krull monoid with class group $G$ such that each class contains a prime divisor. Suppose that $\exp(G)-3 \ge \ro(G)$ and that $G$ does not have a subgroup isomorphic to $C_{\exp(G)}^2$. Then  $\Dast{H}$ is not an interval, as $\exp(G)-3 \notin \Dast{H}$ while $\{1 , \exp(G)-2 \} \subset \Dast{H}$.
\end{theorem}  

The type of groups for which the problem of determining $\Dast{G}$ in more detail has 
received most attention are cyclic groups. In this case $\Dast{G}$ shows a rich structure that is 
not yet fully understood, despite various partial results. 

For $G$ a cyclic group of order $n$ we have, by the results above, that $\max \Dast{G}=n-2$, and it was proved by Geroldinger and Hamidoune \cite{geroldingerhamidoune02} that the second largest element of $\Dast{G}$ 
is $\lfloor n/2 \rfloor -1$ for $n\ge 4$.

Recently several further elements were determined by Plagne and the author \cite{WASdeltacyclic}; we state a simplified version of the result (the actual result goes down to a tenth, rather than a fifth, of the order of the group). 
  
\begin{theorem}
\label{thm_directweak}
Let $G$ be a cyclic group of order at least $n_0$ (where $n_0=250$ is a possible choice).
We have
\begin{equation*}
\begin{split}
\Dast{G} \cap \N_{\ge |G|/5}  =
\N \cap \biggl \{ |G|- & 2, \frac{|G|-2}{2}, \frac{|G|-3}{2}, \frac{|G|-4}{2},
\frac{|G|-4}{3},   \\
 &  \frac{|G|-6}{3},\frac{|G|-4}{4},\frac{|G|-5}{4},\frac{|G|-6}{4},
\frac{|G|-8}{4} \biggr \}.
\end{split}
\end{equation*}
\end{theorem}
 
An important tool in obtaining this result is the determination of $\mD{G_0}$ for 
$G_0$ a set with $|G_0|= 2$. 
The key-case, to which all other cases can be reduced, is that 
$G_0 = \{e, ae\}$ where $e$ is a generating element and $\gcd (a, \ord e )= 1$. 

In this case, one can express $\mD{G_0}$ in terms of the continued fraction 
expansion of $(\ord e  ) /a$. More specifically, one has the following results \cite[Theorem 2.1]{changetal07}. 

\begin{theorem}
Let  $G = \langle e \rangle$ with $\ord e = n >3$. Further, let $a \in [2,n-1]$ and let $[a_0, a_1, \dots, a_m]$ be the continued fraction expansion of $n/a$ of odd length (that is $m$ is even). Then
\[\mD{ \{e,ae\}}= \gcd(a_1, a_3, \dots, a_{m-1}).\]  
\end{theorem}
The continued fraction expansion mentioned in the result is the standard continued fraction expansion, except for the fact that one allows the last term to equal  $1$, which allows to always achieve that $m$ is even.

As a consequence of this, one obtains the following elements that correspond precisely to those $a$ for which the continued fraction expansion has length $3$.  

\begin{remark}
Let  $G = \langle e \rangle$ with $\ord e = n >3$. Further, let  $b,c\in [1,n-1]$ such that $(n-b)/c$ and $(n-b-c)/(bc)$ are positive integers. Then 
\[\min \Delta \left( \left\{e, \frac{n-b}{c}e\right\} \right)= \frac{n-b-c}{bc}.\]
\end{remark}

Moreover, it can be shown that  if  $\mD{\{e,ae\}}$ is  `large' then it must be of that form (cf. \cite[Corollary 3.2]{changetal07} and \cite{WASdeltacyclic}). 
  
\begin{theorem}
\label{aux_thm_2elements}
Let $G$ be a cyclic group, $e$ be a generating element of $G$ and $a\in [1,|G|]$ such that $\gcd(a,|G|)=1$.

Then $\min \Delta(\{e, ae\})> \sqrt{|G|}$ if and only if there exist
some positive integers $c_1$ and $c_2$ such that
\[
a = \frac{|G| - c_1}{c_2}
\]
and the quantity
\[
d_a = \frac{|G| - (c_1 + c_2)}{c_1 c_2}
\]
is integral and satisfies $d_a> \sqrt{|G|}$. Indeed, in this case $\min \Delta(\{e, ae\})=d_a$.
\end{theorem}

These results already explain the presence of several of the elements we mentioned in Theorem \ref{thm_directweak}. Specifically one gets the elements 
\[ \left \{|G| - 2,  \frac{|G|-3}{2}, \frac{|G|-4}{3},   \frac{|G|-5}{4} , \frac{|G|-4}{4} \right \} \]
for $(c_1,c_2)$ equaling $(1,1)$, $(1,2)$, $(1,3)$, $(1,4)$, and $(2,2)$ respectively. 

Furthermore, for every subgroup $G'$ of $G$, one gets that $\exp(G')-2$ is an element of $\Dast{G'}$ and thus of $\Dast{G}$. This yields the elements 
\[ \left \{ \frac{|G|-4}{2},   \frac{|G|-6}{3}, \frac{|G|-8}{4}  \right \},\] considering subgroups of order $|G|/2$, $|G|/3$, $|G|/4$, respectively. In addition,  $(|G|-6)/4$ is in $\Dast{G}$ as  $(\exp(G')-3)/2$ is in  $\Dast{G'}$ for $G'$ a subgroup of order $|G|/2$. 

It remains to construct $\{  (|G|-2)/2 \}$. This element can be shown to equal  $\mD{\{e,-e, (|G|/2)e\}}$. 
In this way we have given some  arguments for the presence of all these elements. 
Of course it remains to show that there are no other elements.
We do not discuss this here. 

For other types of groups the set $\Dast{G}$ is less well-understood. But, it is for example known for $n \ge 5$ that $\{n-3 , n-2\} \subset \Dast{C_n^2}$ and $\max( \Dast{C_n^2} \setminus \{n-3 , n-2\})= \lfloor n/2 \rfloor -1$ (see \cite[Corollary 3.7]{WASchar}). Further results of this form can be obtained for more general groups under assumptions; see \cite[Theorem 3.2]{WASchar} and \cite{geroldingerzhongCHAR}. 
We end with a specific problem and a general remark on further work. 

\begin{problem}
Is there a finite abelian group $G$ such that $\Dast{G}$ is an  interval and $\exp(G) \ge 2\ro(G)+2$?    
\end{problem}

For $n \le 2r+1$, it follows that $\Dast{C_n^r} = [1, \max\{n-2, r-1,\}]$ as 
$[1, r-1]$ and $[\max \{1, n-r-1\}, n-2]$ are contained in it. 

Having some information about the differences  $\Dast{H}$ at hand a next natural question would be to determine which periods can appear in the STSL. Beyond the information contained in the complete results on $\lc(H)$, for special cases which we recalled above, not too much is known on this problem. However, given the recent progress on the problem of determining $\Dast{G}$ and associated descriptions of sets yielding the relevant distances, it might now be a good time to approach this problem.

\subsection{The bound in the STSL }

Having discussed the set of differences we turn to the other parameter in the STSL, the bound. A lot less is known about it. 
Geroldinger and Grynkiewicz \cite[Theorem 4.4.2]{geroldingergrynkiewicz09} showed the following refinement and generalization of Theorem \ref{thm_STSL}. 

\begin{theorem}
\label{thm_STSLfD}
Let $H$ be a Krull monoid with subset of classes containing prime divisors $G_P$ such that $\Do(G_P)$ is finite (and at least $3$). Let 
\[
M = (2 \Do(G_P) - 5 )\Do(G_P)^2 + \frac{1}{2}\Do(G_P)^4)^{\frac{\Do(G_P)(\Do(G_P)-1)}{2}}.
\] 
For each $a \in H$ its set of lengths $\Lo(a)$ is an AAMP with bound $M$ and difference $d \in \Delta(H)$.  
\end{theorem}

The condition that $\Do(G_P)\ge 3$ is no actual restriction as otherwise the monoid is half-factorial. 
As mention in Section \ref{sec_prel} finiteness of $G_P$ implies finiteness of $\Do(G_p)$. 
Thus, the result includes the case that only a finite number of classes contains prime divisors. 
We highlight that in this result the set of differences is $\Delta(H)$ not $\Dast{H}$. 
However, in case the class group is finite we can combine the results to get that  
every set of lengths is an AAMP with difference in $\Dast{H}$ and still have an explicit bound.

The bound above, being of the form $\exp (c  \log (\Do(G_P)) \Do(G_P)^2)$, grows quite fast in terms of the Davenport constant. 
It is not at all clear what the actual order of magnitude of the bound should be. 
Below we give a simple example showing that the dependence is at least of quadratic order.

\begin{example}
Let $n \ge 6$ be even, such that $n/2$ is odd. Let $C_{n/2} \oplus C_n =  \langle e_1 \rangle \oplus \langle e_2 \rangle $.  Then, for sufficiently large $k$, one has that the set of lengths of  $(e_1(-e_1))^{kn/2}(e_2(-e_2))^{kn}$ is an AAP with difference $1$ and bound (at least) $(n-3)(n/2 - 3)$ while  $\Do(C_{n/2} \oplus C_{n})= 3n/2  - 1$.
\end{example}

To see this let $d_1,d_2$ be co-prime positive integers. Then, for all sufficiently large $k_1,k_2$ one has that  $L= (a+ d_1 \cdot [0,k_1])+ (b+ d_2 \cdot [0,k_2])$ is an AAP with difference $1$ and bound (at least)  $(d_1-1)(d_2-1)$; recall that the Frobenius number of $d_1,d_2$ is $(d_1-1)(d_2-1)-1$. Thus $a+b \in L$ while $a+b + (d_1-1)(d_2-1)-1 \notin L$ so that when writing $L= y + ( L' \cup L^{\ast} \cup L'')$ in the usual way with $L^{\ast}$ an AP with difference $1$, that is an interval, then $ y \ge a+b + (d_1-1)(d_2-1)$ and $a+b \ge y-M$ implies that $M \ge (d_1-1)(d_2-1)$.     
Now, by Lemma \ref{lem_AP} the set of length of $(g(-g))^{k \ord g}$ is an AP with difference $\ord g -2$ of length $k$. 
And $\Lo ((e_1(-e_1))^{kn/2}(e_2(-e_2))^{kn}))= (2k + (n/2-2)\cdot [0,k])+  (2k + (n-2) \cdot [0,k])$.
If $n/2$ is odd, $n-2$ and $n/2 -2$ are co-prime. By the argument above we thus have an AAP with bound at least $(n-3)(n/2 - 3)$ in $\lc(C_{n/2} \oplus C_n)$, and $\Do(C_{n/2} \oplus C_{n})= 3n/2  - 1$. 

This example shows that the bound is at least of quadratic order in terms of the Davenport constant.

\begin{problem}
What is the (rough) order of magnitude of the bound in the STSL for $\lc(G)$ (in terms of $\Do(G)$)?
\end{problem}
Initially, it would also be interesting to have an answer to this problem just for some special (infinite) family of groups, 
or in other more restricted scenarios.

There is very little evidence on which one might base conjectures regarding the size of the bound $M$. 
However, an effect that might limit the size of the bound is that elements divisible by prime divisors from many different classes tend to have very simple sets of lengths. We recall a result in this direction due to Geroldinger and Halter-Koch \cite[Theorem 7.6.9]{geroldingerhalterkochBOOK}; their actual result is more precise. 

\begin{theorem}
Let $H$ be a Krull monoid with finite class group, and let $\varphi: H \to F$ be its divisor theory. 
If $a\in H$ such that $\varphi(a)$ is divisible by a prime divisor from each non-zero class, then $\Lo(a)$ is an interval.   
\end{theorem}

\section*{Acknowledgment}
The author is very grateful to the referee for many useful remarks and corrections.


\end{document}